\documentclass{article}
\usepackage{amsmath}
\usepackage{amsthm}
\usepackage{amsfonts}
\usepackage{amssymb}
\usepackage{amscd}
\usepackage[all]{xy}
\usepackage{graphicx}
\usepackage{xcolor}
\usepackage{enumerate}

\font\smallsc=cmcsc10
\font\smallsl=cmsl10

\usepackage{showlabels}

\newtheorem{theorem}{Theorem}

\newtheorem{lemma}[theorem]{Lemma}

\newtheorem{proposition}[theorem]{Proposition}

\newcommand{\Ocal}{\mathcal O}

\newcommand{\C}{\mathbb{C}}
\newcommand{\Z}{\mathbb{Z}}
\newcommand{\Q}{\mathbb{Q}}

\renewcommand{\:}{\colon }

\begin{document}

\title{Abel-Prym maps for isotypical components of Jacobians}
\author{Juliana Coelho\\{\scriptsize julianacoelhochaves@id.uff.br}
\and Kelyane Abreu\\{\scriptsize kelyane\_abreu@hotmail.com}}

\maketitle

\begin{abstract}
Let $C$ be a smooth non-rational projective curve over the complex field $\C$. If $A$ is an abelian subvariety of the Jacobian $J(C)$, we consider the  Abel-Prym map  $\varphi_A : C \rightarrow A$ defined as  the composition of the Abel map of $C$ with the norm map of $A$. 
The goal of this work is to investigate the degree of the map  $\varphi_A$
in the case where $A$ is one of the components of an isotypical decomposition of $J(C)$. 
In this case we obtain a lower bound for $\deg(\varphi_A)$
and, under some hypotheses, also an upper bound.
We then apply the results obtained to compute degrees of Abel-Prym maps in a few examples. In particular, these examples show that both bounds are sharp.
\end{abstract}

\section{Introduction}

Given a smooth  projective complex curve $C$ and a fixed point $q$ in $C$, the Abel map of $C$ is the map taking a point $p$ of $C$ to the point of the Jacobian $J(C)$ of $C$ corresponding to the divisor $p-q$. It is a very well-known fact that the Abel map is an embedding when $C$ is non-rational.
Now, if $A$ is an abelian subvariety of $J(C)$,  composing the Abel map of $C$ with the norm map of $A$ 
we get the Abel-Prym map
 $\varphi_A: C\rightarrow A$. 
It is easy to see that the image of $\varphi_A$ is a subcurve of $A$ generating $A$. The question of finding bounds for the genus of subcurves of abelian varieties has been considered, for  instance, in  \cite{bcv}, \cite{debarre} and \cite{jow}.  However, not much is known in general about the map $\varphi_A$.

There are few  cases in the literature where the degree of $\varphi_A$ is calculated. 
One case that has been studied is when $A$ is a Prym variety of a covering of smooth projective curves. 
The first of such cases was that of a 
double covering $C\rightarrow C'$ that is either \'etale or ramified at two points. In this case, Mumford showed in \cite{mumford} that $\varphi_A$ has degree 2 if $C$ is hyperelliptic, and degree 1 otherwise. 
More recently, Lange and Ortega considered the case where the covering is cyclic,
in \cite{ortega2}, and non-cyclic of degree 3, in \cite{ortega1}. They showed that, under some numerical conditions,  the Abel-Prym map of $A$ has degree 1 or 2.
In \cite{Brambila}, Brambila-Paz, G\'omez-Gonz\'alez and Pioli constructed a Prym-Tyurin variety $A'$  isogenous to $A$  and, again under some conditions, showed that the Abel-Prym map of $A'$ has degree 1. 
In \cite{lrprym}, Lange and Recillas considered Prym varieties associated to  pairs of coverings of smooth projective curves and obtained Abel-Prym maps of degree 1.
Finally, in \cite{PT3}  Lange, Recillas and Rojas used correspondences to construct   a Prym-Tyurin variety $A$ of exponent $3$ having Abel-Prym map of degree 1. 
Note that in all the cases above, the degree of the Abel-Prym map is one or two. 
However,  as our results show, this is not always  the case.

\subsection{Main Results}

The degree of an  Abel-Prym map  seems to be rather difficult to understand in full generality. 
In this work we focus on Abel-Prym maps for Prym-Tyurin components of the isotypical decomposition of a Jacobian variety.

First in Section \ref{sec:genAPmap} we define the Abel-Prym map.  Using a result of Debarre \cite{debarre}, we obtain in Theorem \ref{prop:upperbound} an upper bound on the degree of a map from a curve to a polarized abelian variety, under some conditions:

\smallskip
\noindent{\bf Theorem A} \;
{\it Let $(A,\theta)$ be a polarized abelian variety of dimension $n$  and let  $d=\deg(\theta)$.
Let $\varphi:C\rightarrow A$ be a morphism, where $C$ is a smooth projective complex curve.
Assume that 
$\varphi_*[C]=\dfrac{k}{(n-1)!}\bigwedge^{n-1}\theta$ 
and that $\varphi(C)$ generates $A.$  Then
$$\deg(\varphi) \leq \dfrac{kd}{\sqrt[n]{d}}.$$
}\smallskip

Next, in Section \ref{sec:isot}, we consider a  smooth projective complex curve $C$ with an action of a finite group $G$ and
the isotypical decomposition 
$A_1\times\ldots\times A_n\rightarrow J(C)$
of its Jacobian, as introduced in \cite{decomp}. 
Now set $\varphi_i=\varphi_{A_i}$ and  $C_i=\varphi_i(C)$ so that we have Abel-Prym maps
$$\varphi_i\: C\rightarrow C_i,$$
for each $i=1,\ldots,n$. First we show in Proposition \ref{prop:dim1} that
if $\dim(A_i)=1$ then the degree of $\varphi_i$ is equal to the exponent of $A_i$ in $J(C)$.
To deal with the components of higher dimension we let
  $K_i$ be the kernel of the  	representation of $G$ associated to the component $A_i$,  and consider the quotient curve  $\tilde C_i=C/K_i$ with quotient morphism
$$\psi_i\: C\rightarrow \tilde C_i.$$
Using the map $\psi_i$,
we obtain a lower bound on the degree of the Abel-Prym map $\varphi_i$ and show that, under some conditions, this bound is sharp.
More precisely, in Theorems \ref{propf} and \ref{thm:conj} we show:

\smallskip
\noindent{\bf Theorem B} \;
{\it
With the above notation, 
there is a morphism $f_i\colon \tilde C_i\rightarrow C_i$ such that $\varphi_i=f_i\circ \psi_i$. In particular  
$$\deg(\varphi_i)\geq|K_i|.$$

Moreover, if $A_i$ is a Prym-Tyurin variety of exponent $e(A_i)$ for  $C$, we have:
\begin{enumerate}[(i)]
\item $\deg(f_i)\leq \dfrac{e(A_i)}{|K_i|}$. In particular, if $e(A_i)=|K_i|$ then $f_i$ is a normalization and $\deg(\varphi_i)=|K_i|.$
\item If $\psi_i^*$ is an embedding then $A_i$ is a Prym-Tyurin variety of exponent $\dfrac{e(A_i)}{|K_i|}$ for $\tilde C_i$. Moreover, in this case,   $e(A_i)=|K_i|$ if and only if  $A_i=\psi_i^*(J(\tilde C_i))$.
\end{enumerate}
}\smallskip

In Section \ref{sec:examples} we apply these results to four cases. In the two initial examples, we consider specific actions of $\Z_2$ and $D_4$ on a smooth curve of genus $3$ and $4$, respectively. We compute the degrees of the Abel-Prym maps and show that the bounds obtained in Theorems \ref{prop:upperbound} and \ref{propf} are sharp.
In the last two examples we consider the isotypical decomposition of a Jacobian of a smooth curve with an action of the dihedral group $D_p$, for $p$ an odd prime, and the quaternions group $Q_8$, following \cite{ccr} and \cite{decomp}. We then use our results to compute or give bounds for the degrees of the Abel-Prym maps in these cases, 
under some hypotheses on the action.

\section{Technical background} \label{sec:back}

In this paper we always work over the field of complex numbers $\mathbb{C}$.
For most of the following notations, definitions and results concerning abelian varieties we follow \cite{lange}. 

Let $(A,\theta)$ be a (polarized) abelian variety of dimension $n$. If $(d_1,\ldots,d_n)$ is the \emph{type} of $\theta$ then the \emph{degree} of the polarization $\theta$ is $\deg(\theta)=d_1\cdot\ldots\cdot d_n$.
We denote by $(\hat A,\theta_\delta)$ the \emph{dual abelian variety} of $A$. 
Recall that the map $\Phi_\theta\colon A\rightarrow \hat A$ given by $\Phi_\theta(x)=t_x^*\theta\otimes \theta^{-1}$ is an isogeny of degree $\deg(\Phi_\theta)=(\deg(\theta))^2$, where $t_x$ is the translation by $x$ in $A$.  The \emph{exponent} $e(\theta)$  of $\theta$ is the exponent of the isogeny $\Phi_\theta$, that is, the exponent of the kernel of $\Phi_\theta$.
Denote by $\Psi_\theta\colon \hat A\rightarrow A$ the unique isogeny such that $\Phi_\theta\circ \Psi_\theta=e(\theta)id_A.$

Let $A'$ be an abelian subvariety of $(A,\theta)$, with canonical embedding $i\colon A'\rightarrow A$. The \emph{exponent} of $A'$ in $(A,\theta)$  (or simply in $A$, when there is no possible ambiguity) is $e(A')=e(i^*\theta)$. Moreover, the \emph{norm endomorphism} of $A$ associated to $A'$ is the composition
$N_{A'}=i\circ\Psi_{i^*\theta}\circ \hat i\circ \Phi_\theta,
$
where $\hat i$ is the dual homomorphism of $i$.
By abuse of notation we will also denote by $N_{A'}$ the composition 
$$N_{A'}=\Psi_{i^*\theta}\circ \hat i\circ \Phi_\theta\colon A\rightarrow A'.$$

A \textit{curve} $C$ is a connected,
projective and reduced scheme of dimension $1$ over $\mathbb{C}$. 
The \emph{genus} of $C$ is $g(C):=\dim H^1(C,\Ocal_C)$.
For a smooth curve $C$, we denote by $J(C)$  the \emph{Jacobian variety}  of $C$ and by $\Theta_C$ its theta divisor. 
Recall that the pair $(J(C),\Theta_C)$ is a \emph{principally polarized abelian variety}, that is, $\deg(\Theta_C)=1$. 
Moreover, recall that the \emph{Abel map} $\alpha\colon C\rightarrow J(C)$ is an embedding.

A principally polarized abelian variety $(A,\theta)$ is a \emph{Prym-Tyurin} variety for a smooth curve $C$ if $A$ is (isomorphic to) a subvariety of $J(C)$ and
$i^*\Theta_C\equiv e\theta$, where $i\colon A\rightarrow J(C)$ is the inclusion map and $\equiv$ means algebraic equivalence. Note that in this case $e(A)=e$.

\section{Abel-Prym maps}\label{sec:genAPmap}

Let $C$ be a smooth non-rational curve and let $(J(C),\Theta_C)$ be its principally polarized Jacobian.
If $A$ is an abelian subvariety of $J(C)$, we define the \emph{Abel-Prym map} of $A$, 
$$\varphi: C \rightarrow A,$$
as the composition 
$\varphi=N_A \circ \alpha$
of the Abel map with the norm map of $A$. 

\begin{lemma}\label{lem:CAsubcurve}
Let $C$ be a smooth non-rational curve and let $\varphi: C \rightarrow A$
be the Abel-Prym map to an abelian subvariety  $A$ of $J(C)$. 
Then the image $C_A := \varphi(C)$
is a subcurve of $A$ generating $A$. In particular, $g(C_A)\geq \dim (A)$.
\end{lemma}
\begin{proof}
The first two assertions follow from the fact that  $\alpha(C)$ is isomorphic to $C$ and  generates $J(C)$, and the dual $\widehat{i}_A$ of the inclusion $i_A\colon A\hookrightarrow J(C)$ is surjective onto its image.

For the last assertion we consider the normalization  map $\nu_A\:C_A'\rightarrow C_A$ and the  composition $f_A\:C_A'\rightarrow A$ with the inclusion of $C_A$ in $A$. Since the image of $f_A$ generates $A$, then the map $\tilde f_A\:J(C_A')\rightarrow A$ induced by the universal property is surjective and thus $\dim (A)\leq g(C_A')\leq g(C_A)$.
\end{proof}

By  \cite[Prop. II.6.8]{hartshorne}, the surjective morphism
\begin{eqnarray}\label{eq:varphiA}
\varphi_A:C\rightarrow C_A
\end{eqnarray}
given by $\varphi_A=\varphi$
is finite  and the curve $C_A$ is complete, although possibly singular.
It is not easy in this generality to determine when $C_A$ is smooth, or to compute its genus. 
Under some conditions, we obtain in Theorem \ref{prop:upperbound} an upper bound of the degree of $\varphi_A$. Moreover, in the case where $A$ is  is one of the components of an isotypical decomposition of $J(C)$, we will obtain in Theorem \ref{thm:conj} a lower bound. First, we need a definition.

The \emph{minimal class} of  a polarized abelian variety  $(A,\theta)$  of dimension $n$ is defined as
$$ z_\theta=\frac{1}{(n-1)!}\bigwedge^{n-1}\theta.$$

\begin{lemma}\label{lem:debarre}
Let $C$ be an irreducible curve on a polarized abelian variety $(A,\theta)$ with $\dim(A)=n$ and $\deg(\theta)=d.$ Assume that $C$ generates $A$ and that the class of $C$ in $A$ is $[C]=k\,z_\theta$.
Then 
$$
k\geq\frac{\sqrt[n]{d}}{d}.
$$
\end{lemma}
\begin{proof}
By  \cite[Prop. 4.1]{debarre}, the intersection product of the class $[C]$ with the polarization $\theta$ satisfies
$[C]\cdot \theta \geq n\sqrt[n]{d}$. 
On the other hand, recall that
 by Rieman-Roch and \cite[Prop. 5.2.3]{lange} we have ${\theta^{n}}=d\,{n!}$.
Thus this intersection product is
$[C]\cdot \theta = kdn$ and the result follows.
\end{proof}

\begin{theorem}\label{prop:upperbound}
Let $(A,\theta)$ be a polarized abelian variety of dimension $n$  and let $\varphi:C\rightarrow A$ be a morphism, where $C$ is a smooth curve.
Assume that 
$\varphi_*[C]=k\,z_\theta$ 
and that $\varphi(C)$ generates $A.$  Then
$$\deg(\varphi_A) \leq \dfrac{kd}{\sqrt[n]{d}}$$
where $d=\deg(\theta)$.
\end{theorem}
\begin{proof}
Let $C_A'$ be the normalization of $C_A=\varphi(C)$ and let 
$h: C \rightarrow C_A'$ be the induced morphism. Then 
  $\varphi=g \circ h$,
where 
$g:C_A'\rightarrow A$ is the normalization map of $C_A$ composed with the inclusion of $C_A$ in $A$.
Then $\deg(\varphi_A)=\deg(h)$ and 
$$\varphi_{\ast}[C] 
=g_{\ast}( h_{\ast}[C])
= g_{\ast}(deg(h)[h(C)])
=deg(\varphi_A)g_{\ast}[C_A'].$$ 
Therefore, 
$$g_{\ast}[C_A']=\frac{k}{deg(\varphi_A)}\, z_\theta,$$ 
and since $g_{\ast}[C_A']=[C_A]$, we have by Lemma \ref{lem:debarre}
$$\frac{k}{deg(\varphi_A)}\geq \frac{\sqrt[n]{d}}{d}$$ 
and the result follows.
\end{proof}

Note that if $A$ is a Prym-Tyurin variety of exponent $k$  for $C$ then, by Welter's criterium and the previous proposition, we have $\deg(\varphi_A) \leq k,$ since $\deg(\theta)=1.$

\begin{proposition}\label{prop7}
Let $(A,\theta)$ be a principally polarized abelian variety and let $\varphi:C\rightarrow A$ be a morphism, where $C$ is a smooth curve.	Assume that $\varphi^*$ is an embedding and  
$\varphi_*[C]=k\,z_\theta$.  
Then $\deg(\varphi_A)=k$ if and only if $\varphi(C)$ is smooth and $A=J(\varphi(C))$.
In addition, if $\deg(\varphi_A)=1$ then $A=J(C)$.
\end{proposition}
\begin{proof}
Using notation of the proof of Theorem \ref{prop:upperbound}, we note that 
$\deg(\varphi_A)=k$ if and only if $[C_A]=z_\theta$. By   Matsusaka's criterion (cf. \cite[Remark 12.2.5]{lange}), this happens if and only if $C_A$ is smooth and $A=J(C_A)$. 
The last assertion follows from \cite[Cor. 12.2.6]{lange}.
\end{proof}

The following result is a direct consequence of \cite[Lemma {12.3.1}]{lange}. For the sake of the completeness, we include it here.

\begin{proposition}	\label{prop:morfPT}
Let $f\colon C\rightarrow C'$ be a finite morphism of smooth curves and assume that the pullback $f^*\colon J(C')\rightarrow J(C)$ is an embedding. Then $J(C')$ is a Prym-Tyurin variety of exponent $\deg(f)$ for $C$.
\end{proposition}
\begin{proof}
We'll apply Welters' criterion. Consider the composition 
$$h:=\alpha_{C'}\circ f\colon C\rightarrow J(C'),$$ 
where $\alpha_{C'}$ is the Abel map of $C'$. Then $h^*=f^*\circ \alpha_{C'}^*$ is an embedding, since $\alpha_{C'}^*$ is an isomorphism. Moreover, if $\deg(f)=q$ then $f_*[C]=q[C']$.
This completes the proof, since by  Poincar\'e's formula \cite[Prop. 11.2.1]{lange} we have $[C']=z_{\Theta_{C'}}$.
\end{proof}

We remark that the hypothesis of $f^*$ being an embedding is fundamental in Proposition \ref{prop:morfPT}.
Indeed, consider $f:C\to C'$ a non-ramified morphism of degree $2$ between a curve $C$ of genus $3$ and $C'$ of genus $2$. By \cite[Theorem {12.3.3}]{lange}, the complementary subvariety $A$ of $f^*(J(C'))$ in $J(C)$ is a (classical) Prym variety, hence $A$ is a Prym-Tyurin variety of exponent $2$ for $C$. 
Therefore $f^*(J(C'))$ is not a Prym-Tyurin variety for $C$, since the polarization induced by $\Theta_C$ in $f^*(J(C'))$ is of type $(1,2)$, by \cite[Cor. {12.1.5}]{lange}.
 Note that, in this case, the exponent of $f^*(J(C'))$ as an abelian subvariety of $J(C)$ is 2, hence still equal to the degree of $f$.

One might then think that
the exponent of $f^*(J(C'))$ as an abelian subvariety of $J(C)$ would always equal to the degree of $f$. However, it is easy to see that this is not the case.
Let $f\:C\to C'$ be  a morphism  between a curve $C$ of genus $2$ and $C'$ of genus $1$. By \cite[Thm. 12.3.3(c) and Cor. 11.4.4]{lange}, then $f$ factors as $f=f_e\circ g$ where  $g\:C\rightarrow C'_e$ and $f_e\:C'_e\to C'$ with $C'_e$ of genus 1 and $f_e$  \'etale. Assume $f_e$ non-constant. Then 
the exponent of $f^*(J(C'))$ as an abelian subvariety of $J(C)$ is
$\deg(g)$, which is strictly smaller than $\deg(f)$.

By \cite[Prop. {11.4.3}]{lange}, the pullback map $f^*$ is injective if and only if $f$ does not factor via a cyclic \'etale covering of degree $\geq2$. This implies that, in the case of a cyclic \'etale covering $f\:C\rightarrow C'$, then $J(C')$ is not a Prym-Tyurin variety for $C.$ The next result shows that even if we consider the image $f^*(J(C'))$ of the pullback map, wich is already an abelian subvariety of $J(C)$, then we may still not have a Prym-Tyurin variety for $C$.

\begin{proposition}\label{prop:etale}
	Let $f\: C \rightarrow C'$  be a non-constant cyclic \'etale morphism of smooth curves, where $g(C')\geq 2$ and consider  the inclusion map  $i\:f^*(J(C'))\rightarrow J(C)$. If $\deg(f)\neq a^{g(C')}$ for some $a\in \mathbb{Z}$, then the induced polarization $i^*\Theta_C$ on $f^*(J(C'))$ is not a multiple of a principal polarization.
\end{proposition}
\begin{proof}
For the sake of simplicity, set $\Theta=\Theta_C$,  $\Theta'=\Theta_{C'}$
and $g'=g(C')$.

Note that $f^*$ factors as $f^*=i\circ j$, where $j\:J(C')\rightarrow f^*(J(C'))$ is an isogeny. Set $n=\deg(f)$. 
By \cite[Lemma 12.3.1]{lange}, we have 
$(f^*)^*\Theta\equiv n\Theta'$. 
Hence the type of $(f^*)^*\Theta$ is $(n,\ldots,n)$ and by \cite[Theorem 3.6.1]{lange}, we have $$\chi((f^*)^*\Theta)=(-1)^dn^{g'}$$
for some $d\in\Z$.
Now, $(f^*)^*=j^*\circ i^*$ and by \cite[Cor. 3.6.6]{lange} we have
$$\chi((f^*)^*\Theta)=\chi(j^*(i^*\Theta))=\deg(j)\chi(i^*\Theta).$$
Since $f$ is \'etale of degree $n$, then the isogeny  $j$ also has degree $n$ and we get that
\begin{eqnarray}\label{eq:chi}
\chi(i^*\Theta)=(-1)^dn^{g'-1}.
\end{eqnarray}
	
If $i^*\Theta$ is a multiple of a principal polarization on $f^*(J(C'))$, then it is of  type $(m,\ldots,m)$ for some $m\in\Z$ and again by \cite[Theorem 3.6.1]{lange}, we have $\chi(i^*\Theta)=(-1)^{ d'} m^{g'}$ for some $ d'\in\Z$. Thus, (\ref{eq:chi}) gives us
$$m^{g'}=n^{g'-1}.$$
We need to show this can only happen when $n=a^{g'}$ for some $a\in\Z$. Let $n=p_1^{r_1}\ldots p_k^{r_k}$ be the prime decomposition of $n$. Then $$n^{g'-1} = p_1^{r_1(g'-1)}\ldots p_k^{r_k(g'-1)}$$
and for this to be of the form $m^{g'}$ we must have
$$r_i(g'-1)=s_i g'$$
for some $s_i\in\Z$, for every $i=1,\ldots,k$.
Since $g'$ and $g'-1$ are coprimes, then $g'$ divides $r_i$ and  we may write $r_i=t_ig'$ for some $t_i\in\Z$, for $i=1,\ldots,k$. But then
$$n=(p_1^{t_1}\ldots p_k^{t_k})^{g'}$$
and the result is proven.
\end{proof}

\section{Isotypical decomposition}\label{sec:isot}

In this section we focus on Abel-Prym maps for subvarieties of a Jacobian variety arising from the  isotypical decomposition introduced in \cite{decomp}. First we briefly recall this decomposition, for more details see \cite[Section 13.6]{lange}.

Let $G$ be a finite group acting on a smooth curve $C$. This action induces an action of $C$ on the Jacobian variety $J(C)$ and hence,
an algebra homomorphism  $$\rho: \Q[G] \rightarrow End_\Q (J(C)),$$ where $\Q[G]$ denotes the rational group algebra of $G.$ 
Since $\Q[G]$ is a semi-simple $\Q$-algebra of finite dimension, there is an unique decomposition
$$\Q[G]=Q_1 \times ... \times Q_r,$$ 
where $Q_i$ are simple $\Q$-algebras. Consider the decomposition of the unit element $1 = e_1 + ... + e_r.$ The elements $e_i \in Q_i,$ 
seen as elements of $\Q[G],$ form a set of orthogonal idempotents contained in the center of $\Q[G].$ 

For  $i=1,...,r$ set   $A_i\subset J(C)$ to be the image of $\rho(me_i)$, where $m$ is a positive integer such that  $\rho(me_i)\in End(J(C))$.
Then by \cite[Prop. 13.6.1]{lange} 
 $A_i$ is a $G$-stable abelian subvariety of $J(C)$ with $Hom_G(A_i,A_j)=0$ for $i\neq j.$ Moreover, the addition map induces an isogeny
	$$A_1 \times \cdots \times A_r \rightarrow J(C).$$ 
This decomposition is called the \textit{isotypical decomposition} of $J(C)$ and we write
$$J(C)\sim A_1 \times \cdots \times A_r .$$

The idempotents $e_i$ can be expressed in terms of representations of $G$ as
\begin{equation}\label{eq:ei}
 e_i := \frac{\deg \chi_i}{|G|} \sum_{g \in G} \textup{tr}_{L_i|\Q}(\chi_i(g^{-1}))g,
\end{equation}
where $\chi_i$ is the character of an irreducible representation of $\rho_i\:G\rightarrow GL(V_i)$ for some complex vector space  $V_i$, and 
$L_i := \Q(\chi_i(g), g \in G).$

For each $i=1,...,n,$ let $K_i$ be the kernel of $\rho_i$ and consider the quotient curve $\tilde C_i=C/K_i$ with induced morphism
\begin{eqnarray}\label{eq:psi}
\psi_i\:C\rightarrow \tilde C_i.
\end{eqnarray}
Note that $\tilde C_i$ is smooth, $\psi_i$ has degree $|K_i|$ and, if the action of $G$ on $C$ is known, it is easy to compute the genus of $\tilde C_i$.

\begin{lemma}\label{lem:AsubpsiJC}
	With the above notation,  $A_i$ is an abelian subvariety of  $\psi_{i}^*(J(\tilde C_i))$.
\end{lemma}
\begin{proof}
By \cite[Prop. 5.2]{cr} we have $\psi_{i}^*(J(\tilde C_i))=Im(p_{K_i})$, where
$$p_{K_{i}}=\frac{1}{|K_{i}|}\sum_{k\in K_{i}} k\;\in\Q[G].$$
Now, we have $\chi_i(g^{-1})=\chi_i(kg^{-1}k^{-1})=\chi_i(g^{-1}k^{-1}),$ for all $k\in K_i$ and $ g\in G$, where the first equality follows from the property of the character and the second one from the fact that $K_{i} = ker(\rho_{i}).$ So,
\begin{align*}
p_{K_{i}}e_i 
&= \left( \frac{1}{|K_{i}|}\sum_{k\in K_{i}} k\right) 
\left(  \frac{\deg \chi_i}{|G|} \sum_{g \in G} \textup{tr}_{L_i|\Q} 
(\chi_i(g^{-1}))g\right) \\
&= \frac{\deg \chi_i}{|G||K_{i}|}\sum_{k\in K_{i}}\sum_{g \in G}\textup{tr}_{L_i|\Q} 
(\chi_i(kg^{-1}k^{-1}))kg \\
&= \frac{\deg \chi_i}{|G||K_{i}|}\sum_{g\in G }\sum_{k \in K_{i}}\textup{tr}_{L_i|\Q} 
(\chi_i(g^{-1}k^{-1}))kg,
\end{align*}
and setting $h=kg,$ we have	
\begin{align*}
p_{K_{i}}e_i 
&= \frac{\deg \chi_i}{|G||K_{i}|}\sum_{g\in G }\sum_{h \in K_{i}g}\textup{tr}_{L_i|\Q} (\chi_i(h^{-1}))h\\
&= \frac{|K_{i}|\deg \chi_i}{|G||K_{i}|}\sum_{h \in G}\textup{tr}_{L_i|\Q} (\chi_i(h^{-1}))h.
\end{align*}
Thus $p_{K_{i}}e_i=e_i$ and
$A_i=Im(e_i)\subset Im(p_{K_i})=\psi_{i}^*(J(\tilde C_i)).$
\end{proof}

With the notation of  (\ref{eq:varphiA}), we set $\varphi_i=\varphi_{A_i}$ and $C_i=C_{A_i}$, so that we have
\begin{equation}\label{eq:varphii}
\varphi_i\:C\rightarrow C_i.
\end{equation}

The next result shows that, when $A_i$ is 1-dimensional, the bound in Theorem \ref{prop:upperbound} is achieved.

\begin{proposition}\label{prop:dim1}
With the above notation, if $\dim(A_i)=1$ then $\deg(\varphi_i)=e(A_i)$.
\end{proposition}
\begin{proof}
Since $\dim(A_i)=1$ then $C_i=A_i$ and $A_i$ is a Prym-Tyurin variety of some exponent $k=e(A_i)$ for $C$. 
Hence, by Proposition \ref{prop7}, we have $\deg(\varphi_i)=k$.
\end{proof}

To deal with the higher dimensional components of the isotypical decomposition, it is usefull to understand
the relation between   $\psi_i$ and $\varphi_i$.

\begin{theorem}\label{propf}
With the above notation, there is a morphism $f_i\colon \tilde C_i\rightarrow C_i$ such that $\varphi_i=f_i\circ \psi_i$. In particular  $\deg(\varphi_i)\geq|K_i|$.
\end{theorem}
\begin{proof}
First note that given an action of $G$ on a complex vector space $V$, we have a representation  $\rho\:G\rightarrow GL(V).$ 
The kernel $K$ of this representation is the intersection of the stabilizers of the elements of $V$, that is,  $K=\bigcap G_x$, where  
$G_x= \{g \in G\ |  \ g \cdot x = x\}$ for each $x\in V$. 
Then we have a morphism $V\rightarrow V/K$ mapping each point $x$ to its orbit  
$O_x = \{g \cdot x\ |  \ g \in G \}$. 

In our case, we have a representation $\rho_i\:G\rightarrow GL(V_i)$ with $K_{i}=ker(\rho_i)$. The trivial action of $K_{i}$ on $V_i$ induces a trivial action of $K_{i}$ on $A_{i}.$ Thus, we have compatible actions of $K_{i}$ on $C$ and $A_i$
and   the morphism $\varphi_i: C \rightarrow A_i$ induces a morphism $\overline \varphi_i \colon C/K_i \rightarrow A_{i}/K_i$.
Therefore, as $C/K_{i}= \tilde C_i$,  $A_{i}/{K_i}\cong A_{i}$ and $Im\overline \varphi_i=C_i$, the  morphism $\overline \varphi_i$ becomes $f_i\colon \tilde C_i\rightarrow C_i.$
\end{proof}

\begin{lemma}\label{prop:MultiplicativityOfExponents}
Let $A\subset A'$ be abelian subvarieties of the Jacobian of a smooth curve $C$. 
Assume that $(A',\theta')$ and $(A, \theta)$ are Prym-Tyurin varieties of exponents $e'$ and $e$, respectively, for $C$. Then
$$e=e'\cdot e_{A'}(A),$$
where
$e_{A'}(A)$ is the exponent of $A$ as a subvariety of $A'$.

Moreover, if $(A',\theta')$ is isomorphic to a principally polarized Jacobian $(J(C'),\Theta_{C'})$, then $A$ is a Prym-Tyurin variety for $C'$.
\end{lemma}
\begin{proof}
Denote by $i\:A'\hookrightarrow J(C)$, $j\:A\hookrightarrow J(C)$ and $h\:A\hookrightarrow A'$
the inclusion maps such that $j=i\circ h$. By hypothesis we have
$$i^*\Theta_C \equiv e'\,\theta'
\quad\text{ and }\quad
j^*\Theta_C \equiv e\,\theta.$$
Moreover,
\begin{equation}\label{eq:multexp}
e\theta\equiv j^*\Theta_C\equiv h^*(i^*\Theta_C)\equiv h^*(e'\theta')\equiv e'h^*\theta',
\end{equation}
which implies that 
$$h^*\theta'\equiv \frac{e}{e'}\theta.$$
Setting $k=\dfrac{e}{e'},$ by \cite[Lemma 6.2]{lan} we have $$e_{A'}(A)=e(h^{*}\theta)=e(k\theta)=ke(\theta)=k.$$
For the last statement, by virtue of (\ref{eq:multexp}) it is enough to note that since $i^*$ and $j^*$ are embeddings, so is $h^*$.
\end{proof}

\begin{theorem}\label{thm:conj}
With the above notation, assume $A_i$ is a Prym-Tyurin variety of exponent $e(A_i)$ for  $C.$ We have:
\begin{enumerate}[(i)]
\item $\deg(f_i)\leq \dfrac{e(A_i)}{|K_i|}$. In particular, if $e(A_i)=|K_i|$ then $f_i$ is a normalization and $\deg(\varphi_i)=|K_i|$.
\item If $\psi_i^*$ is an embedding then $A_i$ is a Prym-Tyurin variety of exponent $\dfrac{e(A_i)}{|K_i|}$ for $\tilde C_i$. Moreover, in this case,   $e(A_i)=|K_i|$ if and only if  $A_i=\psi_i^*(J(\tilde C_i))$.
\end{enumerate}
\end{theorem}
\begin{proof}
(i) Since $A_i$ is Prym-Tyurin for $C$, $\deg(\varphi_i)\leq e(A_i),$ by Theorem \ref{prop:upperbound}. 
By Theorem \ref{propf},
 $\deg(\varphi_i)=\deg(f_i)\deg(\psi_i)$, and since $\deg(\psi_i)=|K_i|$, we have
 $\deg(f_i)\leq \dfrac{e(A_i)}{|K_i|}.$ Moreover, if $e(A_{i})=|K_{i}|$ then $\deg(f_{i})\leq1$ and thus $\deg(f_{i})=1.$

(ii) First note that since $\psi_i^*$ is an embedding then
$J(\tilde C_i)\cong \psi_i^*(J(\tilde C_i))$
and, by Proposition \ref{prop:morfPT}, $J(\tilde C_i)$
 is a Prym-Tyurin variety of exponent $|K_i|$ for $C$. Hence, by Lemma \ref{prop:MultiplicativityOfExponents}, $A_i$ is a Prym-Tyurin variety of exponent $\dfrac{e(A_i)}{|K_i|}$ for $\tilde C_i$. 
Moreover, if $e(A_i)=|K_i|$, then $A_i$ is a Prym-Tyurin variety of exponent 1 for $\tilde C_i$ and hence, by \cite[Cor. {12.2.6}]{lange}, we must have $A_i\cong J(\tilde C_i)$, which  implies $A_i= \psi_i^*(J(\tilde C_i))$.
\end{proof}

We remark that our results rely strongly on the hypothesis that  $A_i$ is Prym-Tyurin for $C$.
In order to understand Abel-Prym maps for non Prym-Tyurin subvarieties of $J(C)$, it seems important first to investigate whether the bound in Theorem \ref{prop:upperbound} still holds in the case where 
$(\varphi_i)_*[C]$ is not a multiple of the minimal class.

Finally, in \cite[Section 2]{decomp}, a further decomposition of the Jacobian is introduced, where each factor $A_i$ is (non-canonically) decomposed into a product $B_i^{n_i}$ of isogenous abelian subvarieties of $A_i$. It would be interesting to study the Abel-Prym map to  these $B_i$ and compare it with the Abel-Prym map to $A_i$. We plan to address this issue in a future work.

\section{Examples}\label{sec:examples}

In this section we use the notation of Section \ref{sec:isot}.

\subsection{Action of $\Z_2$ over a curve of genus 3}

Consider $\Z_{2}=\{\overline{0},\overline{1}\}$ the cyclic group of order 2 acting on a smooth curve $C$ of genus $3$,  such that the induced action on the Jacobian variety of $C$, is given by
$$\begin{array}{llll}
\overline{1}(\alpha_1)=\alpha_3, & \overline{1}(\alpha_2)=\alpha_2 & \text{and} & \overline{1}(\alpha_3)=\alpha_1;\\
\overline{1}(\beta_1)=\beta_3, & \overline{1}(\beta_2)=\beta_2 & \text{and} & \overline{1}(\beta_3)=\beta_1;
\end{array}$$
where
$\alpha_i,\beta_i$  with $i=1,...,3$ is the usual sympletic basis of $J(C)$.

Let us present the isotypical decomposition of $J(C)$. 
The character table of $\mathbb{Z}_2$, associated the trivial representation $\rho_0$ and  the sign representation $\rho_1,$ is given by 

\begin{center}
	\begin{tabular}{|c|c|c|c|}
		\hline
		& $\overline{0}$ & $\overline{1}$ \\ \hline
		$\chi_{0}$ & 1 & 1   \\ \hline
		$\chi_{1}$ & 1 & -1   \\ \hline
	\end{tabular}
\end{center}
By (\ref{eq:ei}) we have 
$$e_0=\frac{1}{2}(\overline{0}+\overline{1}) 
\quad \text{and} \quad
e_1=\frac{1}{2}(\overline{0}-\overline{1}).$$
and the isotypical decomposition is $J(C)\sim A_0 \times A_1.$

The abelian variety $A_0$ is given by 
$$\frac{\langle2\epsilon_1,2\epsilon_2,2\delta_1,2\delta_2\rangle_\mathbb{C}} {\langle2\epsilon_1,2\epsilon_2,2\delta_1,2\delta_2\rangle_\mathbb{Z}}$$\\
with 
$$2\epsilon_1={\alpha_1+\alpha_3}, 
\quad
\epsilon_2=\alpha_2, 
\quad
2\delta_1={\beta_1+\beta_3}
\quad \text{and} \quad
\delta_2=\beta_2.$$ 
Now, $A_0$ is a 2-dimensional variety of type $(1 \ 2)$ and hence it is not Prym-Tyurin for $C.$ 
Since  $\rho_0$ is the trivial representation, we have  $K_0=\mathbb{Z}_2$ and hence $\deg(\psi_0)=|K_0|=2$.
Hence, by Theorem \ref{propf}
 $$\deg(\varphi_{0})=2deg(f_{0})\geq 2.$$

The abelian variety $A_1$ has  dimension 1 and type $(2)$ and hence, by Proposition \ref{prop:dim1}, $\deg(\varphi_1)=2.$


\subsection{Action of $D_4$ over a curve of genus 4}

Consider $D_{4}=\langle a,b \ | \ a^{4} = b^{2} = (ab)^{2} = 1\rangle$  the dihedral group of order 8 acting on a smooth curve $C$ of genus $4$
such that the induced action on the Jacobian variety of $C$, is given by
$$\begin{array}{lllll}
a(\alpha_1)=\alpha_2, & a(\alpha_2)=\alpha_3, & a(\alpha_3)=\alpha_4 &\text{and} & a(\alpha_4)=\alpha_1;\\
a(\beta_1)=\beta_2, & a(\beta_2)=\beta_3, & a(\beta_3)=\beta_4 & \text{and} & a(\beta_4)=\beta_1;\\
b(\alpha_1)=-\alpha_2, & b(\alpha_2)=-\alpha_1, & b(\alpha_3)=-\alpha_4 & \text{and} & b(\alpha_4)=-\alpha_3;\\
b(\beta_1)=-\beta_2, & b(\beta_2)=-\beta_1, & b(\beta_3)=-\beta_4 & \text{and} & b(\beta_4)=-\beta_3.
\end{array}$$
where
$\alpha_i,\beta_i$  with $i=1,...,4$ is the usual sympletic basis of $J(C)$.

Let us present the isotypical decomposition of $J(C).$ There are four irreducible representations of degree one and just one of degree two. The character table of $D_4$, associated to the irreducible representations $\rho_0$,..., $\rho_4$ (see \cite{groups} or \cite{ccr}) is given by 

\begin{center}
	\begin{tabular}{|c|c|c|c|c|c|c|}
		\hline
		& 1 & $\{a^{2}\}$ & $\{a, a^{3}\}$ & $\{b, a^{2}b\}$ & $\{ab, a^{3}b\}$ \\ \hline
		$\chi_{0}$ & 1 & 1 & 1 & 1 & 1  \\ \hline
		$\chi_{1}$ & 1 & 1 & 1 & -1 & -1  \\ \hline
		$\chi_{2}$ & 1 & 1 & -1 & 1 & -1  \\ \hline
		$\chi_{3}$ & 1 & 1 & -1 & -1 & 1  \\ \hline
		$\chi_{4}$ & 2 & -2 & 0 & 0 & 0  \\ \hline	
	\end{tabular}
\end{center}
Again by (\ref{eq:ei}), we have 
$$e_0=\frac{1}{8}(1+a+a^2+a^3+b+ab+a^2b+a^3b),$$
$$e_1=\frac{1}{8}(1+a+a^2+a^3-b-ab-a^2b-a^3b),$$
$$e_2=\frac{1}{8}(1-a+a^2-a^3+b-ab+a^2b-a^3b),$$
$$e_3=\frac{1}{8}(1-a+a^2-a^3-b+ab-a^2b+a^3b)$$
and
$$e_4=\frac{1}{4}(1-2a^2).$$
Hence the isotypical decomposition is  
$$J(C)\sim A_0 \times A_1 \times A_2 \times A_3 \times A_4.$$ 

The varieties $A_0$ and $A_3$ are both trivial. 
The abelian varieties $A_1$ and $A_2$ both have dimension 1 and type $(4)$ and thus, by Proposition \ref{prop:dim1}, we have $\deg(\varphi_1)=\deg(\varphi_2)=4$.

Now, $A_4$ is a 2-dimensional variety of type $(5 \ 5)$ given by 
$$\dfrac{\langle4\epsilon_1,4\epsilon_2,4\delta_1,4\delta_2\rangle_\mathbb{C}} {\langle4\epsilon_1,4\epsilon_2,4\delta_1,4\delta_2\rangle_\mathbb{Z}},$$ \\
with 
$$4\epsilon_1={\alpha_1-2\alpha_3}, 
\quad
4\epsilon_2={\alpha_2-2\alpha_4},
\quad
4\beta_1={\beta_1-2\beta_3}, 
\quad \text{and} \quad
4\beta_2={\beta_2-2\beta_4}.$$ 
The irreducible representation $\rho_4$ has trivial kernel and 
consequently, $\deg(\psi_4)=1$ and $\tilde{C_4}=C.$ Moreover, as $A_4$ is Prym-Tyurin of exponent $e(A_4)=5$ for $C$, by \cite[Lemma 1.3]{Brambila} $\deg(\varphi_4)$ divides $e(A_4).$ Thus $\deg(\varphi_4)=1$ or $5.$ 
Assume first that $\deg(\varphi_4)= 5.$ Then by Proposition \ref{prop7}, $C_4$ is smooth and $A_4=J(C_4).$ Hence $g(C_4)=dim(A_4)=2$ and, by Riemman-Hurwitz theorem applyed to $\varphi_4,$ we have 
$$2g(C)-2=deg(\varphi_4)(2g(C_4)-2)+R$$
which gives $R=-4,$ an absurd.
Thus, $\deg(\varphi_4)\neq 5$ and we must have  $$\deg(\varphi_4)=1.$$
In particular, $\deg(\varphi_4)=|K_4|$ and the bound in Theorem \ref{propf} is achieved.

\subsection{Action of $D_p$, for $p$ an odd prime} 

Let $p$ be an odd prime and consider the dihedral group 
$$D_p=\langle r,s \ |  \  r^p=s^2=(rs)^2=1\rangle$$ 
acting on a smooth curve $C$. Then by \cite[Theorem 6.4]{ccr} we have 
$$J(C) \sim A_0 \times A_1 \times A_2,$$ where the $D_p$-action on each factor is given by the trivial action $\rho_0$ on $A_0,$ the alternating action $\rho_1$ on $A_1$,
and a degree-2 irreducible representation $\rho_2$ on $A_2$.

We assume that for some involution $\kappa$ in $D_p$ 
the  quotient map $C\rightarrow C/\langle\kappa\rangle$ is not a cyclic \'etale covering. Then by \cite[Proposition 11.4.3]{lange}, 
the pullback map $J(C/\langle\kappa\rangle)\rightarrow J(C)$ is an embedding.

Assume $A_0$ non-trivial. Now, since $\rho_0$ is the trivial action, we have $K_0=D_p$ and $\tilde C_0=C/D_p$. By 
\cite[Theorem 6.4]{ccr}, $A_0$ is isomorphic to $J(\tilde C_0)$. Hence by \cite[Proposition 5.1]{ccr} and Proposition \ref{prop:morfPT}, $A_0$ is a Prym-Tyurin variety of exponent $2p$ for $C$ and, by Theorem \ref{thm:conj}, we have
$$\deg(\varphi_0)=2p.$$

Now assume that $A_1$ is non-trivial. By \cite[Section 4]{ccr}, we have $K_1=\langle r\rangle$ and hence $\tilde C_1=C/\langle r\rangle$ and $|K_1|=p$. Then by Theorem \ref{propf}, 
$$\deg(\varphi_1)\geq p.$$ 
Now consider the commutative diagram
$$\begin{array}{ccc}
J(\tilde C_0) & \rightarrow & J(C/\langle k\rangle)\\
\downarrow & & \downarrow\\
J(\tilde C_1) & \rightarrow & J(C).
\end{array}$$
By \cite[Proposition 5.1]{ccr}, the map $J(\tilde C_0) \rightarrow  J(C/\langle k\rangle)
$ is injective and, by hypothesis $J(C/\langle\kappa\rangle)\rightarrow J(C)$ is also injective. Hence
$\psi_1^*\colon J(\tilde C_1) \rightarrow  J(C)$ is injective. 
Now, by \cite[Theorem 6.4]{ccr}, $A_1$ is the pullback to $J(C)$ of the Prym variety $P$ associated to 
$\tilde C_1\rightarrow \tilde C_0$, that is, 
$P$ is the complement of the image of the map $J(\tilde C_0)\rightarrow J(\tilde C_1)$. Since $\tilde C_1\rightarrow \tilde C_0$ is a degree-2 map, by \cite[Theorem 12.3.3]{lange}, if this map is either unramified or ramified in exactly two points, then  $P$ is a Prym-Tyurin variety of exponent 2 for $\tilde C_1$. Thus, in this case $A_1$ is a Prym-Tyurin variety of exponent $2p$ for $C$ and, by Theorem \ref{prop:upperbound}, we have 
$$p\leq \deg(\varphi_1)\leq 2p.$$ 

Finally, the irreducible representation $\rho_2$ has trivial kernel and $A_2$ is not a Prym-Tyurin variety for $C$, in general. Thus, we know nothing about the degree of the map $\varphi_2.$

We have shown:

\begin{proposition} Let $C$ be a smooth curve admiting an action of the dihedral group $D_p$, for $p$ be an odd prime. 
Assume that for some involution $\kappa\in D_p$ 
the  quotient map $C\rightarrow C/\langle\kappa\rangle$ is not a cyclic \'etale covering. 
If the factors $A_0$ and $A_1$ of the isotypical decomposition of $J(C)$ associated to the 
 trivial and the alternating actions, respectively, are non-trivial, then:
\begin{itemize}
\item $\deg(\varphi_0)=2p$;
\item $\deg(\varphi_1)\geq p$. Moreover, if $\tilde C_1\rightarrow \tilde C_0$ is either unramified or ramified in exactly two points, then $p\leq \deg(\varphi_1)\leq 2p.$
\end{itemize}
\end{proposition}

\subsection{Action of the quarternion group $Q_8$} 

Let $Q_8=\{\pm1,\pm i,\pm j,\pm ij\ |\ i^2=j^2=(ij)^2=-1,\, (-1)^2=1\}$ be the quaternion group acting on a smooth curve $C$. By
\cite[Proposition 5.2]{decomp}, the isotypical decomposition of $J(C)$ with respect to this action is
$$J(C) \sim A_0 \times A_1 \times A_2\times A_3\times A_4,$$ 
where the $Q_8$-action on each factor is given by the trivial action on $A_0$, and the actions described in \cite[Section 4.3]{decomp} on the other factors.

Assume $A_0$ non-trivial. Then $K_0=Q_8$  and $A_0$ is isomorphic to $J(\tilde C_0)$, where $\tilde C_0=C/Q_8$. By Theorem \ref{propf} we have
$$\deg(\varphi_0)\geq 8.$$ 
Moreover, we have equality if  $\psi_0:C\rightarrow \tilde C_0$ does not factor through a cyclic \'etale cover, by 
Proposition \ref{prop:morfPT} and Theorem \ref{thm:conj}.

As for $A_1$, $A_2$ and $A_3$, we have by \cite[Section 4.3]{decomp} that
$K_1=\langle i\rangle$, $K_2=\langle j\rangle$ and $K_3=\langle ij\rangle$.
Thus for $\ell=1,2,3$ we have that $|K_\ell|=4$, and $\psi_0$ factors as the composition
$$C\rightarrow C/\langle-1\rangle \rightarrow \tilde C_\ell \rightarrow \tilde C_0$$
where $\tilde C_\ell=C/K_\ell$.

Assume $A_\ell$ non-trivial. Then by Theorem \ref{propf}, 
$$\deg(\varphi_\ell)\geq 4.$$
Moreover, by \cite[Proposition 5.2]{decomp} $A_\ell$ is the pullback to $J(C)$ of 
the Prym variety $P$ associated to the map $\tilde C_\ell\rightarrow \tilde C_0$, that is
$P$ is the complement of the image of the map $J(\tilde C_0)\rightarrow J(\tilde C_\ell)$. 
By \cite[Theorem 12.3.3]{lange} if $\tilde C_\ell\rightarrow\tilde C_0$ is either unramified or ramified in exactly two points then $P$ is a Prym-Tyurin variety of exponent 2 for $\tilde C_\ell$ and hence $A_\ell$ is a Prym-Tyurin variety of exponent 8 for $C$. In this case, by Theorem \ref{prop:upperbound}, 
$$4\leq \deg(\varphi_\ell)\leq 8.$$

Finally, assume $A_4$ non-trivial. Since $K_4=\langle 1\rangle$, we have $\tilde C_4=C$ and we cannot use Theorem \ref{propf} to obtain a lower bound on the degree of $\varphi_5$. On the other hand, by \cite[Proposition 5.2]{decomp} $A_4$ is the Prym variety associated to the degree-2 map 
$C\rightarrow C/\langle -1\rangle$, that is, the complement of the image of the pullback $J(C/\langle -1\rangle)\rightarrow J(C)$. Thus if this map is either unramified or ramified in exactly two points, then $A_4$ is a Prym-Tyurin variety of exponent 2 for $C$ and by Theorem \ref{prop:upperbound} we have
$$\deg(\varphi_4)\leq2.$$

We have shown:

\begin{proposition}
Let $C$ be a smooth curve admiting an action of the quaternion group $Q_8$. If 
$$J(C) \sim A_0 \times A_1 \times A_2\times A_3\times A_4$$ 
is the isotypical decomposition of the Jacobian of $C$, ordered as in \cite[Proposition 5.2]{decomp}, then
$$\deg(\varphi_0)\geq 8\quad\text{and}\quad \deg(\varphi_\ell)\geq 4$$
for $\ell=1,2,3$. Moreover, we have
\begin{itemize}
\item If  $C\rightarrow \tilde C_0$ does not factor through a cyclic \'etale cover, then $\deg(\varphi_0)=8$;
\item For $\ell=1,2,3$, if the map $\tilde C_\ell \rightarrow \tilde C_0$ is either unramified or ramified in exactly two points, then $4\leq \deg(\varphi_\ell)\leq 8$;
\item If the map $C\rightarrow C/\langle -1\rangle$ is either unramified or ramified in exactly two points, then $\deg(\varphi_4)\leq2.$
\end{itemize}
\end{proposition}

\section*{Acknowledgement}
The authors wish to thank Anita Rojas for many interesting discussions on the subject of this paper.
\smallskip 

\noindent This paper is part of the second author's thesis written at the Universidade Federal Fluminense, and was financed in part by the Coordena\c c\~ao de Aperfei\c coamento de Pessoal de N\'ivel Superior - Brasil (CAPES) - Finance Code 001.

\bigskip
\noindent{\smallsc Juliana Coelho, Universidade Federal Fluminense (UFF), 
Instituto de Matem\'atica e Estat\'istica - 
 Rua Prof. Marcos Waldemar de Freitas Reis, S/N, 
 Gragoat\'a, 24210-201 - Niter\'oi -  RJ,  Brasil}\\
{\smallsl E-mail address: \small\verb?julianacoelhochaves@id.uff.br?}
\bigskip
\bigskip

\noindent{\smallsc Kelyane Abreu, Universidade Federal Fluminense (UFF), 
Instituto de Matem\'atica e Estat\'istica - 
 Rua Prof. Marcos Waldemar de Freitas Reis, S/N, 
 Gragoat\'a, 24210-201 - Niter\'oi -  RJ,  Brasil}\\
{\smallsl E-mail address: \small\verb?kelyane_abreu@hotmail.com?}

\end{document}